\documentclass[11pt, a4paper]{amsart}
\usepackage[top=35truemm,bottom=35truemm,left=35truemm,right=35truemm]{geometry}
\usepackage{setspace}
\usepackage{amsmath, amsfonts,amsthm,amssymb,mathrsfs,amscd}
\usepackage{enumerate}
\usepackage[dvipdfmx]{graphicx}
\usepackage{ascmac}
\usepackage[arrow,matrix]{xy}
\usepackage{color}
\usepackage{array}
\usepackage{pifont}
\usepackage{longtable}
\usepackage{here}
\usepackage{braket}
\usepackage{mathtools}

\usepackage{tikz}

\newtheorem{theo}{Theorem}[section]
\newtheorem*{theo*}{Theorem}
\newtheorem{defi}{Definition}[section]
\newtheorem{exa}{Example}[section]

\newtheorem{cor}{Corollary}[section]
\newtheorem{lem}{Lemma}[section]

\newtheorem{introthm}{Theorem}

\providecommand{\abs}[1]{\left\lvert#1\right\rvert}

\makeatletter
\@addtoreset{equation}{section}

\makeatother

\def\G{\mathcal{G}}
\def\Z{\mathbb{Z}}
\def\F{\mathbb{F}}

\def\N{\mathbb{N}}

\def\R{\mathbb{R}}

\def\({\left(}
\def\){\right)}
\def\<{\langle}
\def\>{\rangle}

\begin{document}

\title[ ]{The space of marked Dyer systems, monotonicity, and continuity of growth rates}
\author{Tomoshige Yukita}
\address{Liberal Arts Education Center, Ashikaga University, Omaecho 2680-1,Ashikaga,Tochigi, 326-8558, Japan}
\email{yukita.tomoshige@g.ashikaga.ac.jp}
\date{}
\thanks{}

\begin{abstract}
The space $\mathcal{G}_n$ of $n$-marked groups provides a natural framework for studying algebraic and geometric invariants under deformation.
In general, the growth rate is not continuous on $\mathcal{G}_n$.
In this paper, we investigate the subspace $\mathcal{D}_n \subset \mathcal{G}_n$ consisting of $n$-marked Dyer systems,
which extend Coxeter systems and include graph products of cyclic groups and right-angled Artin groups.
We prove that $\mathcal{D}_n$ is closed in $\mathcal{G}_n$ and introduce a natural partial order on $\mathcal{D}_n$ with respect to which the growth rate is monotonically increasing.
As a consequence, the growth rate function $\tau : \mathcal{D}_n \to \mathbb{R}_{\ge 1}$ is continuous.
The proof combines the solution to the word problem for Dyer systems by Paris and Soergel, the parabolic growth formula by Paris and Varghese, and analytic arguments based on normal convergence and Hurwitz's theorem.
This extends the continuity results known for Coxeter systems to the broader class of Dyer systems.
\end{abstract}

\maketitle
%%%%%%%%%%%%%%%%%%%%%%%%%%%%%%%%%%%%%%%%%%%
%%%%%%%%%%%%%%%%%%%%%%%%%%%%%%%%%%%%%%%%%%%
\section{Introduction}
The space $\mathcal{G}_n$ of $n$-marked groups,
introduced by Grigorchuk \cite{Grigorchuk1984},
provides a natural framework for studying groups from a deformation-theoretic viewpoint.
A point of $\mathcal{G}_n$ is an isomorphism class of a pair $(G, S)$ consisting of a group $G$ and an ordered generating set $S=(s_1,\dots,s_n)$.
Fixing a free group $\F_n$ of rank $n$ with a fixed free marking,
each $(G, S)$ corresponds to an epimorphism $\F_n \twoheadrightarrow G$,
and convergence in $\mathcal{G}_n$ reflects stabilization of relations in bounded word length.
Thus $\mathcal{G}_n$ may be regarded as a deformation space of marked groups of rank $n$.

\smallskip
Within this space, it is natural to investigate how algebraic and geometric invariants behave under convergence.
Among such invariants, the growth rate
\[
\tau(G,S)=\lim_{m\to\infty}\sqrt[m]{b_{(G, S)}(m)}
\]
plays a fundamental role,
where $b_{(G, S)}(m)$ is the number of elements of $G$ whose word lengths with respect to $S$ are at most $m$.
Geometrically,
$\tau(G, S)$ is the volume entropy of the Cayley graph associated with the marked group $(G, S)$.
Since the growth rate depends on the chosen generating set,
it should be viewed as an invariant of the \emph{marked object},
rather than of the abstract group $G$.

\smallskip
In general, the growth rate function $\tau:\mathcal{G}_n \to \mathbb{R}_{\ge 1}$ is not continuous.
Nevertheless, from the deformation viewpoint, it is natural to ask whether geometric constraints imply continuity on the growth rate.
In the hyperbolic setting, Fujiwara and Sela \cite{FujiwaraSela2023} proved that the growth rate varies continuously with the marking and that the set of growth rates of a fixed hyperbolic group is well-ordered.
Fujiwara \cite{Fujiwara2025} extended this continuity phenomenon to certain acylindrically hyperbolic groups.
These results suggest that geometry may enforce continuity of growth rates within suitable subspaces.

\smallskip
In the context of Coxeter systems, continuity phenomena are also closely related to geometric convergence of hyperbolic reflection groups.
Work of Floyd \cite{Floyd1992} and Kolpakov \cite{Kolpakov2012} on sequences of hyperbolic Coxeter polyhedra in dimensions two and three demonstrates continuity of growth rates under geometric convergence.

\smallskip
For Coxeter systems, continuity of growth rates on the space of marked Coxeter systems was established in \cite{Yukita2024}.
Two structural ingredients play a decisive role there:
the monotonicity of growth rates under increases of Coxeter parameters, proved by Terragni \cite{Terragni2016},
and Steinberg's formula expressing the reciprocal of the growth series as an alternating sum over parabolic subgroups \cite{Davis2025}.
The guiding principle is that \emph{monotonicity combined with analytic control of growth series forces continuity of the growth rate}.

\smallskip
Dyer systems, introduced by Dyer \cite{Dyer1990}, form a natural extension of Coxeter systems.
They include Coxeter groups, graph products of cyclic groups, and right-angled Artin groups as special cases.
Recent advances provide the two key structural tools required for growth analysis:
Paris and Soergel \cite{ParisSoergel2023} established a solution to the word problem via a theory of reduced syllabic words,
while Paris and Varghese \cite{ParisVarghese2024} proved a parabolic decomposition formula for growth series, generalizing Steinberg's formula.
These developments suggest that Dyer systems constitute the largest natural class in which the Coxeter strategy for growth can be extended.

\smallskip
The aim of this paper is to study the subspace $\mathcal{D}_n \subset \mathcal{G}_n$
consisting of $n$-marked Dyer systems and to analyze the behavior of growth rates on this space.
Our main results are the following.

\begin{introthm}[Compactness of $\mathcal{D}_n$]
The space $D_n$ is closed in $\mathcal{G}_n$.
\end{introthm}

\begin{introthm}[Monotonicity of growth rates]
If $(D,S) \preceq (D',S')$ in the natural partial order on Dyer systems,
then $a_{(D,S)}(m) \le a_{(D',S')}(m)$ for all $m \ge 0$.
In particular, $\tau(D,S)\le \tau(D',S')$.
\end{introthm}

\begin{introthm}[Continuity of growth rates]
The growth rate function $\tau : D_n \longrightarrow \mathbb{R}_{\ge 1}$ is continuous.
\end{introthm}

The compactness of $D_n$ is obtained by identifying it with the compact space of Dyer matrices,
and by exploiting Soergel’s realization of Dyer systems as finite-index subgroups of Coxeter systems.
Monotonicity follows from the Paris-Soergel theory of reduced syllabic words.
Continuity is deduced from the Paris-Varghese growth formula, normal convergence of associated rational functions on the unit disk, and an application of Hurwitz's theorem.

\smallskip
The organization of this paper is as follows.
In Section \ref{section:2},
we review the space $\mathcal{G}_n$ of marked groups and basic facts on growth functions and growth rates,
including convergence criteria and auxiliary lemmas on subgroups.
In Section \ref{section:3},
we recall the space of marked Coxeter systems and summarize known results on its topology.
In Section \ref{section:4},
we introduce Dyer graphs, Dyer matrices, and the space $\mathcal{D}_n$ of marked Dyer systems,
and prove its compactness together with a characterization of convergence.
In Section \ref{section:5},
we establish monotonicity of growth rates with respect to the natural partial order on Dyer systems.
Finally, in Section \ref{section:6}, we prove the continuity of the growth rate function on $\mathcal{D}_n$.

%%%%%%%%%%%%%%%%%%%%%%%%%%%%%%%%%%%%%%%%%%%
%%%%%%%%%%%%%%%%%%%%%%%%%%%%%%%%%%%%%%%%%%%
\section{The space of marked groups and the growth rates}\label{section:2}
In this section, we recall the space of marked groups (see \cite{ChampetierGuilardel2005, Grigorchuk1984} for further details).
For the reader's convenience, we include proofs of several well-known facts.

\smallskip
For a group $G$ and its ordered finite generating set $S=(s_1, \cdots, s_n)$, 
the pair $(G, S)$ is called an \emph{$n$-marked group}.
Two $n$-marked groups $(G, S)$ and $(H, T)$ are said to be \emph{isomorphic} if the map $S\ni s_i\mapsto t_i\in T$ extends to a group isomorphism $G\to H$.
The set of all isomorphism classes of $n$-marked groups is denoted by $\G_n$ and is called the \emph{space of $n$-marked groups}.

\begin{exa}
Let $\F_n$ denote the free group of rank $n$.
For any two ordered free generating sets $X=(x_1,\dots,x_n)$ and $Y=(y_1,\dots,y_n)$ of $\F_n$,
the $n$-marked free groups $(\F_n, X)$ and $(\F_n, Y)$ are isomorphic.
Such an ordered free generating set is called a free marking of $\F_n$.
\end{exa}

In what follows,
we do not distinguish between free markings of $\F_n$ and fix a free marking $X$.
For any $n$-marked group $(G, S)$, we denote by $\pi_{(G, S)}:\F_n\twoheadrightarrow G$ the epimorphism sending $x_i\in X$ to $s_i\in S$, which is called the \emph{marking of $(G, S)$}.
We also call the ordered finite generating set $S$ the \emph{marking}.
It is straightforward to verify that two $n$-marked groups $(G, S)$ and $(H, T)$ are isomorphic if and only if $\ker{\pi_{(G, S)}}=\ker{\pi_{(H, T)}}$.

\smallskip
Let $(G, S)$ be an $n$-marked group. 
For any $x\in{G}$, 
the word length of $x$ with respect to $S$, denoted by $\abs{x}_S$, is defined by 
\begin{equation*}
\abs{x}_S=\min{\Set{n\geq 0|x=s_1\cdots s_n, \ s_i\in{S\cup{S^{-1}}}}}. 
\end{equation*}
We denote by $B_{(G, S)}(R)$ the ball of radius $R$ centered at $1_G$, that is, 
\begin{equation*}
B_{(G, S)}(R)=\Set{x\in{G}|\abs{x}_S\leq R}. 
\end{equation*}

For abbreviation, we write $B(R)$ instead of $B_{(\F_n, X)}(R)$. 
We define a metric $d$ on $\G_n$ as follows. 
For two $n$-marked groups $(G, S)$ and $(H, T)$, 
let us denote by $v\((G, S), (H, T)\)$ the maximal radius of the balls centered at $1_{\F_n}$ where $\ker{\pi_{(G, S)}}=\ker{\pi_{(H, T)}}$, that is, 
\begin{equation*}
v\((G, S), (H, T)\)=\max{\Set{R\geq{0}|B(R)\cap{\ker{\pi_{(G, S)}}}=B(R)\cap{\ker{\pi_{(H, T)}}}}}. 
\end{equation*}
Then the metric $d$ on $\G_n$ is defined by $\displaystyle d\((G, S), (H, T)\)=e^{-v\((G, S), (H, T)\)}$.
It is known that the metric space $(\G_n, d)$ is compact. 
For convenience,
we record the following characterization of convergence in $\mathcal{G}_n$.

\begin{lem}\label{lem:convergence criteria}
A sequence $\{(G_k, S_k)\}_{k\geq 1}$ of $n$-marked groups converges to $(G, S)$ if and only if for every $R\in \N$, there exists $k_0\in \N$ such that for all $k\geq k_0$,
\[
B(R)\cap \ker{\pi_{(G_k, S_k)}}=B(R)\cap \ker{\pi_{(G, S)}}.
\]
\end{lem}

We next prove a lemma concerning the convergence of subgroups, which will be used to show that the space of marked Dyer systems is compact.

\begin{lem}\label{lem:convergence of subgroups}
Let $\{(G_k, S_k)\}_{k\geq 1}$ be a convergent sequence of $n$-marked groups and let $\displaystyle \lim_{k\to \infty}(G_k, S_k)=(G, S)$.
Denote the markings of $(G_k, S_k)$ and $(G, S)$ by $\pi_k$ and $\pi$, respectively.
Let $H<\F_n$ be a finitely generated subgroup with an ordered finite generating set $Y=(y_1,\dots,y_m)$.
Then the sequence of $m$-marked groups $\{(\pi_k(H), \pi_k(Y))\}$ converges to $(\pi(H), \pi(Y))$ in the space $\mathcal{G}_m$ of $m$-marked groups.
\end{lem}
\begin{proof}
To avoid confusion with the fixed free marking X of $\F_n$, 
we denote by $Z = (z_1, \dots, z_m)$ a fixed ordered free generating set of $\F_m$.
Throughout the proof, we use the following notation:
\begin{itemize}
  \item $W_k = \pi_k(H)$ with ordered generating set $T_k = \pi_k(Y)$;
  \item $W = \pi(H)$ with ordered generating set $T = \pi(Y)$;
  \item $\rho_k : \F_m \twoheadrightarrow W_k$, the marking of $(W_k, T_k)$;
  \item $\rho : \F_m \twoheadrightarrow W$, the marking of $(W, T)$;
  \item a group homomorphism $\iota : \F_m \to \F_n$ defined by $z_i \mapsto y_i$.
\end{itemize}

We show that for any $R\in \N$, there exists $k_0 \in \mathbb{N}$ such that for all $k \ge k_0$,
\[
B_{(\F_m,Z)}(R) \cap \ker \rho_k=B_{(\F_m,Z)}(R) \cap \ker \rho.
\]

Fix $R\in \N$.
Set
\[
C = \max\{ |y_1|_X, \ldots, |y_m|_X \}.
\]
Then $|\iota(w)|_X \le CR$ for all $w \in B_{(F_m,Z)}(R)$.
Since $(G_k, S_k) \to (G, S)$, there exists $k_0 \in \mathbb{N}$ such that for all $k \ge k_0$,
\[
B_{(\F_n,X)}(CR) \cap \ker \pi_k=B_{(\F_n,X)}(CR) \cap \ker \pi .
\]

For $w \in \F_m$ with $|w|_Z \le R$, we have $\iota(w) \in \ker \pi_k$ if and only if $\iota(w) \in \ker \pi$ for all $k \ge k_0$.
Since $\rho_k = \pi_k \circ \iota$ and $\rho = \pi \circ \iota$, it follows that for all $k \ge k_0$,
\[
\begin{aligned}
B_{(\F_m,Z)}(R) \cap \ker \rho_k
&= B_{(\F_m,Z)}(R) \cap \iota^{-1}(\ker \pi_k) \\
&= B_{(\F_m,Z)}(R) \cap \iota^{-1}(\ker \pi) \\
&= B_{(\F_m,Z)}(R) \cap \ker \rho .
\end{aligned}
\]
The conclusion follows from Lemma \ref{lem:convergence criteria}.
\end{proof}

For any $n$-marked group $(G, S)\in{\G_n}$,
the \emph{growth function} $b_{(G, S)}(m)$ is defined as the number of elements of $G$ whose word length is at most $m$, namely, 
\begin{equation*}
b_{(G, S)}(m)=\#\Set{x\in{G}|\abs{x}_S\leq m}. 
\end{equation*}
Since $b_{(G, S)}(m)$ is submultiplicative,
Fekete's Lemma (see \cite[p.183]{Harpe2000}) implies that 
\begin{equation*}
\lim_{m\to \infty}\sqrt[m]{b_{(G, S)}(m)}=\inf_{m\geq 0}\sqrt[m]{b_{(G, S)}(m)}. 
\end{equation*}
We define the \textit{growth rate}  $\tau(G, S)$ by
\[
\tau(G, S)=\lim_{m\to \infty}\sqrt[m]{b_{(G, S)}(m)}.
\]
An $n$-marked group $(G, S)$ is said to have \emph{exponential growth} if $\tau(G, S)>1$.
It is known that the property having exponential growth does not depend on the choice of markings.

\smallskip
Let $(G, S)$ be a marked group.
The \emph{growth series $f_{(G, S)}(z)$} of $(G, S)$ is defined by
\begin{equation*}
f_{(G, S)}(z)=\sum_{m\geq 0} a_{(G, S)}(m) z^m,
\end{equation*}
where $a_{(G, S)}(m)$ is the number of elements of $G$ of word length $m$,
so that
\[
a_{(G, S)}(m)=b_{(G, S)}(m)-b_{(G, S)}(m-1) \ (m\geq 1), \ a_{(G, S)}(0)=1.
\]
If $G$ is finite,
then $f_{(G, S)}(z)$ is a polynomial and called the \textit{growth polynomial} of $(G, S)$.
If $G$ is infinite,
by the Cauchy-Hadamard theorem,
we see that the growth rate $\tau(G, S)$ equals the reciprocal of the radius of convergence of the growth series $f_{(G, S)}(z)$.

\begin{exa}\label{exa:growth series of cyclic groups}
Let $C_p$ be the cyclic group of order $p$,
and let $r$ be a cyclic generator of $C_p$.
For finite $p$,
the growth polynomial $f_{(C_p, \{r\})}(z)$ is given by 
\[
f_{(C_p, \{r\})}(z)=\begin{cases}
1+2z+\dots+2z^m & (p=2m+1)\\
1+2z+\dots+2z^{m-1}+z^m & (p=2m)
\end{cases}.
\]
For $p=\infty$,
the growth series $f_{(C_\infty, \{r\})}(z)$ is given by
\[
f_{(C_\infty, \{r\})}(z)=1+2z+2z^2+\dots=1+\sum_{m\geq 0}2z^m=\dfrac{1+z}{1-z}.
\]
\end{exa}
If a growth series $f_{(G, S)}(z)=\dfrac{P(z)}{Q(z)}$ where $P(z)$ and $Q(z)$ are relatively prime polynomials with integer coefficients,
then the reciprocal $\tau^{-1}(G, S)$ is a zero of $Q(z)$ which is smallest among the absolute values of zeros of $Q(z)$.

\begin{defi}
Let $(G, S)$ and $(H, T)$ be two marked groups with $S=(s_1,\dots,s_n)$ and $T=(t_1,\dots,t_m)$.
We define the direct product $(G, S)\times (H ,T)$ as an $(n+m)$-marked group by
\[
(G, S)\times (H ,T)=(G\times H, (s_1,\dots,s_n,t_1,\dots,t_m)),
\]
where we identify elements $s_i, t_j$ with $(s_i, 1_H), (1_G, t_j)$ for $1\leq i\leq n$ and $1\leq j\leq m$.
\end{defi}
It is easy to see that $f_{(G, S)\times (H, T)}(z)=f_{(G, S)}(z)f_{(H, T)}(z)$.
In particular,
\[
\tau((G, S)\times(H, T))=\max{\{\tau(G, S), \tau(H, T)\}}.
\]
This observation will be used in the proof of the continuity of the growth rate $\tau:\mathcal{D}_n\to \R_{\geq 1}$.

%%%%%%%%%%%%%%%%%%%%%%%%%%%%%%%%%%%%%%%%%%%
%%%%%%%%%%%%%%%%%%%%%%%%%%%%%%%%%%%%%%%%%%%
\section{The space of marked Coxeter systems and growth rates}\label{section:3}
In this section, we recall the necessary background on Coxeter systems and summarize known results concerning the topology of the space of $n$-marked Coxeter systems (see \cite{Yukita2024} for further details).

\smallskip
Let $\widehat{\N}=\N\cup{\{\infty\}}$ be the one-point compactification of $\N$.
It follows that a convergent sequence $\{m_k\}_{k\geq 1}$ of $\widehat{\N}$ is constant or strictly increasing.

\smallskip
Let $W$ be a group generated by a finite set $S=\{s_1,\dots,s_n\}$.
The pair $(W,S)$ is called a \emph{Coxeter system} if $W$ admits a presentation of the form
\[
W=\langle s_1,\dots,s_n \mid (s_is_j)^{m_{ij}}=1 \text{ for } 1\le i,j\le n\rangle,
\]
where $m_{ii}=1$ and $m_{ij}\in\widehat{\mathbb{N}}_{\ge 2}$.
In the case $m_{ij}=\infty$, the relation $(s_is_j)^{m_{ij}}=1$ is omitted.
%The \emph{rank} of a Coxeter system $(W,S)$ is defined to be the cardinality $\#S$ of $S$.
Such a generating set is called a \emph{Coxeter generating set}, and a group that admits a Coxeter generating set is called a \emph{Coxeter group}.
When the Coxeter generating set $S$ is equipped with a total order, the pair $(W,S)$ is called a \emph{marked Coxeter system}.

\smallskip
For a graph $\Lambda$, we denote its vertex set and edge set by $V(\Lambda)$ and $E(\Lambda)$, respectively.
A graph is said to be \emph{simple} if it contains neither multiple edges nor loops.
When $\Lambda$ is simple, an edge $e\in E(\Lambda)$ is written as $e=\{u,v\}$, where $u$ and $v$ are the endpoints of $e$.
A \emph{Coxeter graph} $(\Lambda,m)$ consists of a simple graph $\Lambda$ together with a map $m\colon E(\Lambda)\to\widehat{\mathbb{N}}_{\ge 3}$, called the \emph{edge weight}.
If the vertex set is endowed with a total order, the Coxeter graph is said to be \emph{marked}.
Since the weight $3$ occurs most frequently, it is omitted from all figures in this paper.
By abuse of notation, we write $\Lambda$ for a Coxeter graph $(\Lambda,m)$.

\smallskip
Marked Coxeter systems and marked Coxeter graphs correspond to each other as follows.
For an $n$-marked Coxeter system $(W,S)$ with $S=(s_1,\dots,s_n)$, we define the marked Coxeter graph $\Lambda(W,S)$ by setting $V(\Lambda)=(v_1,\dots,v_n)$, where two vertices $v_i$ and $v_j$ are joined by an edge of weight $m_{ij}$ if $m_{ij}\ge 3$.
Conversely, given a marked Coxeter graph $\Lambda=(\Lambda,m)$ with $V(\Lambda)=(v_1,\dots,v_n)$, the marked Coxeter system $(W(\Lambda),S)$ is defined by the presentation
\[
W(\Lambda)=\left\langle s_1,\dots,s_n \,\middle|\,
\begin{array}{l}
s_1^2=\cdots=s_n^2=1,\\
(s_is_j)^2=1 \quad \text{if } \{v_i,v_j\}\notin E(\Lambda),\\
(s_is_j)^{m(\{v_i,v_j\})}=1 \quad \text{if } \{v_i,v_j\}\in E(\Lambda)
\end{array}
\right\rangle.
\]
These two constructions are mutually inverse.
The marked Coxeter graph $\Lambda(W, S)$ (respectively, the marked Coxeter system $(W(\Lambda), S)$) obtained in this manner is called the \emph{marked Coxeter graph associated with} $(W, S)$ (resp. the \emph{marked Coxeter system associated with} $\Lambda$).

\smallskip
The set of all $n$-marked Coxeter systems is denoted by $\mathcal{C}_n$,
which forms a subspace of the space $\mathcal{G}_n$ of $n$-marked groups.
This space $\mathcal{C}_n$ is called the \emph{space of $n$-marked Coxeter systems}.

\begin{theo}[{\cite[Theorem 3.6]{Yukita2024}}]\label{theo:Yukita2024 Theorem 3.6}
The space $\mathcal{C}_n$ of $n$-marked Coxeter systems is compact.
Moreover, convergence in $\mathcal{C}_n$ can be characterized as follows.
Let $\{(W_k, S)\}_{k\ge 1}$ be a sequence of $n$-marked Coxeter systems, and let $\Lambda_k=(\Lambda_k, m_k)$ be the Coxeter graph associated with $(W_k, S)$.
Let $(W, S)$ be an $n$-marked Coxeter system, and let $\Lambda=(\Lambda, m)$ be its associated Coxeter graph.
Then $\{(W_k, S)\}_{k\ge 1}$ converges to $(W, S)$ if and only if $\Lambda_k=\Lambda$ as simple graphs for all sufficiently large $k$, and
$\displaystyle \lim_{k\to\infty} m_k(\{v_i,v_j\}) = m(\{v_i,v_j\})$ for every edge $\{v_i,v_j\}$.
\end{theo}

A Coxeter system $(W, S)$ is said to be \emph{irreducible} if the associated Coxeter graph $\Lambda$ is connected.
For a subset $T\subset S$, the subgroup $W_T\subset W$ generated by $T$ is called a \emph{Coxeter subgroup}.
It is well known that the pair $(W_T, T)$ is again a Coxeter system, and that the full
subgraph $\Lambda_T\subset \Lambda$ induced by $T\subset S$, equipped with the restricted weight function, is precisely the Coxeter graph associated with $(W_T, T)$.
Note that the full subgraph $\Lambda_T$ is a subgraph whose vertex set is $T$, in which two vertices are joined by an edge if and only if they are joined by an edge in $\Lambda$.
Consequently, if the Coxeter graph $\Lambda$ decomposes as a disjoint union
\[
\Lambda=\Lambda_1\sqcup\cdots\sqcup\Lambda_N,
\]
then the Coxeter group $W$ decomposes as a direct product
\[
W=W_1\times\cdots\times W_N,
\]
where each $W_i$ is the Coxeter group defined by $\Lambda_i$.

\smallskip
Let $\mathbb{S}^d$ and $\mathbb{E}^d$ denote the $d$-dimensional spherical and Euclidean geometries, respectively.
A Coxeter system $(W, S)$ is called \emph{spherical} (respectively, \emph{Euclidean}) if $W$ (respectively, if $W$ is infinite and) acts properly discontinuously by isometries
on $\mathbb{S}^d$ (respectively, on $\mathbb{E}^d$), with the elements of $S$ acting as reflections.
Irreducible spherical and Euclidean Coxeter systems are completely classified in terms
of Coxeter graphs, which are listed in Figures \ref{fig:irreducible sphericals} and \ref{fig:irreducible Euclidean}.
\begin{figure}[htbp]
\centering
\includegraphics[scale=0.5]{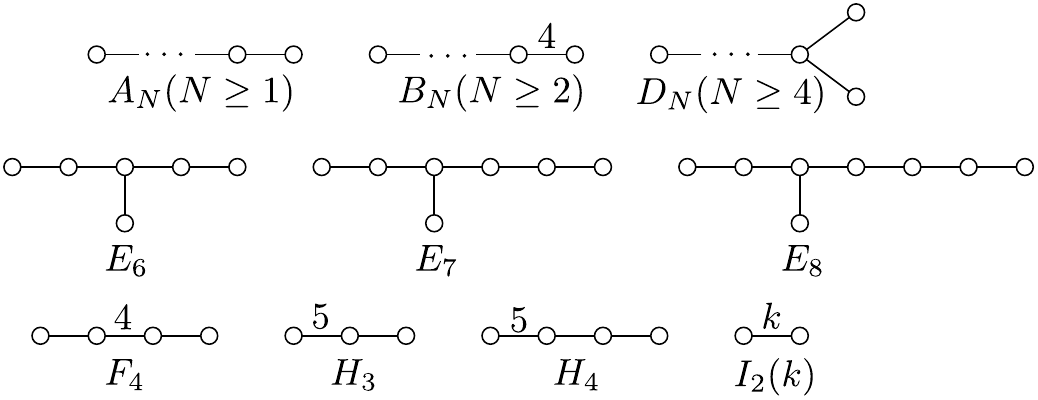}
\caption{The Coxeter graphs of irreducible spherical Coxeter systems with $N$ vertices}
\label{fig:irreducible sphericals}
\end{figure}
\begin{figure}[htbp]
\centering
\includegraphics[scale=0.5]{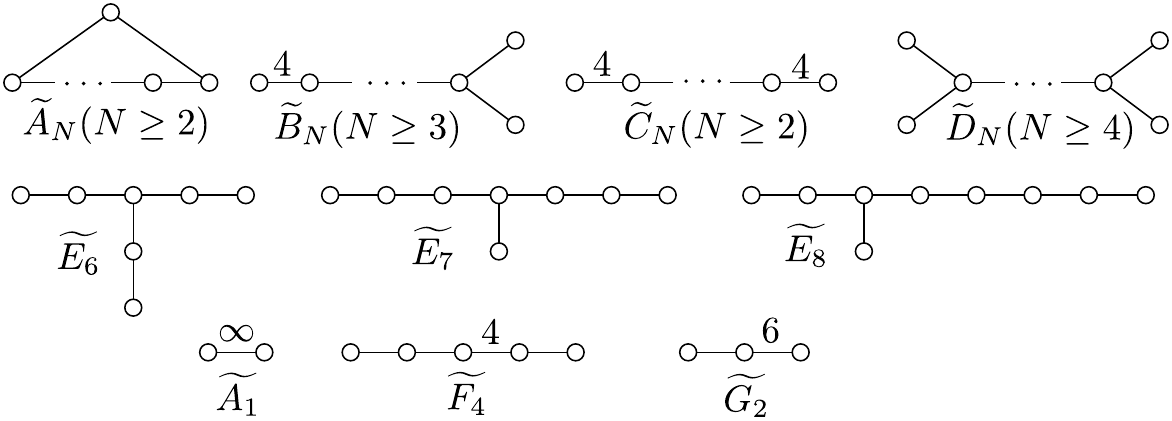}
\caption{The Coxeter graphs of irreducible Euclidean Coxeter systems with $N+1$ vertices}
\label{fig:irreducible Euclidean}
\end{figure}

\begin{theo}[{\cite[Lemma 3.8 and Corollary 3.12]{Yukita2024}}]\label{theo:clopenness of spherical and Euclidean}
The subspace of spherical or Euclidean $n$-marked Coxeter systems is both open and
closed in $\mathcal{C}_n$.
\end{theo}

%%%%%%%%%%%%%%%%%%%%%%%%%%%%%%%%%%%%%%%%%%%
%%%%%%%%%%%%%%%%%%%%%%%%%%%%%%%%%%%%%%%%%%%
\section{The space of marked Dyer systems}\label{section:4}
In this section,
we give the definitions of Dyer graphs, Dyer systems, Dyer matrices, and the space $\mathcal{D}_n$ of Dyer systems of rank $n$.
Then we shall see that $\mathcal{D}_n$ is a compact subspace of $\G_n$. 
The general reference here is \cite{ParisVarghese2024,Soergel2024}.

\smallskip
Let $D$ be a group generated by a finite set $S=\{s_1,\dots,s_n\}$.
Set $[a, b]_k=\underbrace{aba\dots}_k$ for $a,b\in D$ and $k\in \N_{\geq 2}$.
For example,
the relation $[a, b]_2=[b, a]_2$ is the commuting relation $ab=ba$ between $a$ and $b$,
and the relation $[a, b]_k=[b, a]_k$ between involutions $a$ and $b$ is rewritten as $(ab)^k=1$.
The pair $(D, S)$ is called a \emph{Dyer system} if $D$ admits a presentation of the form
\[
D=\left\langle s_1,\dots,s_n \,\middle|\,
s_i^{f_i}=1, \ \left[s_i, s_j\right]_{m_{ij}}=\left[s_j, s_i\right]_{m_{ij}}\text{ for }1\leq i\leq j\leq n
\right\rangle,
\]
where $f_i, m_{ij}\in\widehat{\mathbb{N}}_{\ge 2}$ for $1\leq i,j\leq n$ such that $m_{ij}=2$ whenever $f_i\geq 3$ or $f_j\geq 3$.
In the case $f_i=\infty$ and $m_{ij}=\infty$, the relations $s_i^{f_i}=1$ and $\left[s_i, s_j\right]_{m_{ij}}=\left[s_j, s_i\right]_{m_{ij}}$ are omitted, respectively.
Such a generating set is called a \emph{Dyer generating set}, and a group that admits a finite Dyer generating set is called a \emph{Dyer group}.
When the Dyer generating set $S$ is equipped with a total order, the pair $(D, S)$ is called a \emph{marked Dyer system}.
It follows that $(D, S)$ is a Coxeter system if $f_i=2$ for $1\leq i\leq n$.

\smallskip
Next,
we introduce Dyer graphs, which are useful to represent Dyer systems.
\emph{Note that the definition of Dyer graphs in this paper is not the same as in \cite{ParisVarghese2024, Soergel2024}.}
The definition is modified to fit our objectives.
A \emph{Dyer graph} $(\Gamma, f, m)$ is the triple consisting of a simple finite graph $\Gamma$,
a map $f:V(\Gamma)\to \widehat{\N}_{\geq 2}$, called the \emph{vertex weight},
a map $m:E(\Gamma)\to \widehat{\N}_{\geq 3}$, called the \emph{edge weight},
such that $m(e)=\infty$ for any edge $e=\{u, v\}\in E(\Gamma)$ with $f(u)\geq 3$ or $f(v)\geq 3$(see Figure \ref{fig:example of Dyer graph}).
If $f(v)=2$ for any vertex $v$, we may consider the Dyer graph $(\Gamma, f, m)$ as the Coxeter graph $(\Gamma, m)$.
If the vertex set is endowed with a total order, the Dyer graph is said to be \emph{marked}.
Since the edge weight $3$ occurs most frequently, it is omitted from all figures in this paper.
By abuse of notation, we write $\Gamma$ for a Dyer graph $(\Gamma, f, m)$.

\smallskip
Marked Dyer systems and marked Dyer graphs correspond to each other as follows.
For an $n$-marked Dyer system $(D,S)$ with $S=(s_1,\dots,s_n)$, we define the marked Dyer graph $\Gamma(D,S)$ by $V(\Gamma)=(v_1,\dots,v_n)$, where the vertex weight of $v_i$ is $f_i$ and two vertices $v_i$ and $v_j$ are joined by an edge of weight $m_{ij}$ if $m_{ij}\ge 3$.
Conversely, given a marked Dyer graph $\Gamma=(\Gamma, f, m)$ with $V(\Gamma)=(v_1,\dots,v_n)$, the marked Dyer system $(D(\Gamma),S)$ is defined by the presentation
\[
D(\Gamma)=\left\langle s_1,\dots,s_n \,\middle|\,
\begin{array}{l}
s_1^{f(v_1)}=\cdots=s_n^{f(v_n)}=1,\\[0.2em]
\left[s_i, s_j\right]_2=\left[s_j, s_i\right]_2\text{ if } \{v_i,v_j\}\notin E(\Gamma),\\[0.2em]
\left[s_i, s_j\right]_{m(\{v_i, v_j\})}=\left[s_j, s_i\right]_{m(\{v_i, v_j\})}\text{ if } \{v_i,v_j\}\in E(\Gamma)
\end{array}
\right\rangle,
\]
where the relations $s_i^{f(v_i)}=1$ and $\left[s_i, s_j\right]_{m(\{v_i, v_j\})}=\left[s_j, s_i\right]_{m(\{v_i, v_j\})}$ are omitted if $f(v)=\infty$ and $m(\{v_i, v_j\})=\infty$, respectively.
These two constructions are mutually inverse (see Example \ref{exa:example of Dyer graph and group}).
The marked Dyer graph $\Gamma(D, S)$ (respectively, the marked Dyer system $(D(\Gamma), S)$) obtained in this manner is called the \emph{marked Dyer graph associated with} $(D, S)$ (resp. the \emph{marked Dyer system associated with} $\Gamma$).
It follows that for a Dyer graph $(\Gamma, f, m)$,
\begin{itemize}
\item[(i)]
if $f(v)=2$ for any  $v\in V$,
then the associated Dyer system $(D(\Gamma), S)$ is a Coxeter system,
\item[(ii)]
if $m(e)=\infty$ for any $e\in E$,
then the associated Dyer group $D(\Gamma)$ is a graph product of cyclic groups
\item[(iii)]
if $f(v)=\infty$ for any $v\in V$,
then the associated Dyer group $D(\Gamma)$ is a right-angled Artin group.
\end{itemize}

\begin{exa}\label{exa:example of Dyer graph and group}
Consider a marked Dyer graph $\Gamma$ as in Figure \ref{fig:example of Dyer graph},
where the labels on the vertices and edges are the vertex and edge weights,
respectively.
Note that the edges without labels have weight $3$.
The associated marked Dyer system $D(\Gamma)$ is defined by the following presentation.
\[
D(\Gamma)=\left\langle s_1, s_2, s_3, s_4\,\middle|\,
\begin{array}{l}
s_2^2=s_3^2=s_4^k=1\\[0.2em]
s_1s_3=s_3s_1, \ s_1s_4=s_4s_1, \ s_2s_4=s_4s_2\\[0.2em]
[s_2, s_3]_p=[s_3, s_2]_p
\end{array}
\right\rangle.
\]
\begin{figure}[htbp]
\centering
\includegraphics[scale=0.5]{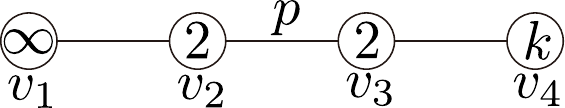}
\caption{Example of Dyer graph $\Gamma$}
\label{fig:example of Dyer graph}
\end{figure}
\end{exa}

Soergel exhibited that any Dyer group $D$ is a finite index subgroup of a Coxeter group $W$ by giving the associated Coxeter graph $\Lambda=(\Lambda, \mu)$ from the associated Dyer graph $(\Gamma, f, m)$ as follows (see \cite{Soergel2024} for details).
Set $V=V(\Gamma), V_2=\Set{v\in V|f(v)=2}, V_p=\Set{v\in V|3\leq f(v)<\infty}$, and $V_\infty=\Set{v\in V|f(v)=\infty}$.
Take copies of $V_p$ and $V_\infty$ denoted by $V'_p=\Set{v'|v\in V_p}$ and $V'_\infty=\Set{v'|v\in V_\infty}$,
respectively.
Note that we distinguish $V'_p$(resp. $V'_\infty)$ from $V_p\subset V$(resp. $V_\infty \subset V$).
The vertex set $V(\Lambda)$ is the disjoint union $V\sqcup V'_p\sqcup V'_\infty$.
We define the edges of $\Lambda$ by the following rules.
\begin{itemize}
\item[(i)]
Two vertices $v, w\in V\subset V(\Lambda)$ with $\{v, w\}\in E(\Gamma)$ are joined by an edge with the edge weight $\mu(\{v, w\})=m(\{v, w\})$.
\item[(ii)]
A vertex $v\in V_p'\sqcup V'_\infty\subset V(\Lambda)$ is only joined to $v'\in V'_p\sqcup V'_\infty$ by an edge with the edge weight $\mu(\{v, v'\})=f(v)$.
\end{itemize}
The Coxeter graph $\Lambda$ is said to be \emph{induced by $\Gamma$}(see Figure \ref{fig:induced Coxeter graph}).
\begin{figure}[htbp]
\centering
\includegraphics[scale=0.5]{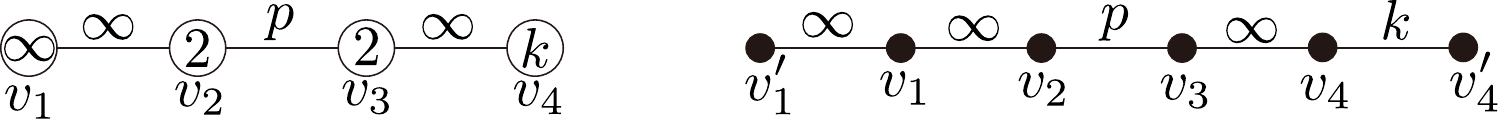}
\caption{Example of induced Coxeter graph: the left-hand side is a Dyer graph $\Gamma$, and the right-hand side is the induced Coxeter graph $\Lambda$}
\label{fig:induced Coxeter graph}
\end{figure}

\begin{theo}[{\cite[Theorem 2.8]{Soergel2024}}]\label{theo:Soergel2024 Theorem2.8}
Take a Dyer graph $\Gamma=(\Gamma, f, m)$ and the induced Coxeter graph $\Lambda$. 
We define a map $\phi:V(\Gamma)\to V(\Lambda)$ by 
\[
\phi(v)=\begin{cases} v & (v\in V_2) \\ vv' & (v\in V_p\sqcup V_\infty) \end{cases}
\]
and set $\Delta=\left\langle \phi(V(\Gamma)) \right\rangle<W(\Lambda)$.
Then $\phi$ extends to the group isomorphism $\phi:D(\Gamma)\to \Delta$, and the index of $\Delta$ is $2(\#V_p+\#V_\infty)$.
\end{theo}

\begin{exa}
In Figure \ref{fig:induced Coxeter graph},
the subgroup $\Delta=\langle v_1v'_1, v_2, v_3, v_4v'_4 \rangle$ is an index $4$ subgroup of the Coxeter group $W(\Lambda)$ associated with $\Lambda$ and is isomorphic to the Dyer group $D(\Gamma)$ associated with $\Gamma$.
\end{exa}

\begin{defi}
A symmetric $n\times n$ matrix $M=(m_{ij})\in M_n(\widehat{\N}_{\geq 2})$ is called a Dyer matrix if $m_{ii}\geq 2$ and $m_{ij}=2$ or $\infty$ for any $j$ if $m_{ii}\geq 3$.
The latter condition reflects the fact that generators of order $\geq 3$ commute only trivially.
The set of all $n\times n$ Dyer matrices is denoted by $\mathcal{DM}_n$,
which is compact as a subspace of the product topological space $(\widehat{\N}_{\geq 2})^{n^2}$.
\end{defi}

For a Dyer matrix $M\in \mathcal{DM}_n$,
we define the \emph{marked Dyer graph $\Gamma_M$ associated with $M$} as follows:
the vertex set is an ordered $n$-element set, say $V=(v_1,\dots,v_n)$, whose vertex weight is defined by $f(v_i)=m_{ii}$.
Two vertices $v_i, v_j$ are joined by an edge of weight $m_{ij}$ if $m_{ij}\geq 3$.
The marked Dyer system associated with $\Gamma_M$ is denoted by $(D_M, S)$.
Conversely, the Dyer matrix $D_\Gamma$ associated with a Dyer graph $\Gamma$ is obtained as follows:
set $V(\Gamma)=(v_1,\dots,v_n)$.
Then $D_\Gamma$ is an $n\times n$ symmetric matrix whose the $(i, j)$-th entry is defined to be $f(v_i)$ for $i=j$, $m(\{v_i, v_j\})$ for $i\neq j$ and $\{v_i, v_j\}\in E(\Gamma)$, and $2$ for $i\neq j$ and $\{v_i, v_j\}\not\in E(\Gamma)$.
It follows that the map $\Phi:\mathcal{DM}_n\ni M\mapsto (D_M, S)\in \mathcal{D}_n$ is a bijection.

\begin{theo}
The map $\Phi$ is a homeomorphism.
In particular,
$\mathcal{D}_n$ is compact.
\end{theo}
\begin{proof}
Let $\{M_k\}_{k\geq 1}$ be a convergent sequence of $n\times n$ Dyer matrices and let $\displaystyle \lim_{k\to \infty}M_k=M$.
We write $\Gamma_k$ and $\Gamma$ (resp. $(D_k, S)$ and $(D, S)$) for the marked Dyer graphs (resp. systems) associated with $M_k$ and $M$.

\smallskip
The idea of the proof is as follows.
Consider the marked Coxeter graphs $\Lambda_k, \Lambda$ induced by $\Gamma_k ,\Gamma$, and the marked Coxeter systems $(W_k, T)$ and $(W, T)$ associated with $\Lambda_k, \Lambda$,
respectively.
Since $(W_k, T)\to (W, T)$ as $k\to \infty$ by Theorem \ref{theo:Yukita2024 Theorem 3.6},
we see that the marked subgroup $\Delta_k$ of $W_k$ obtained as in Theorem \ref{theo:Soergel2024 Theorem2.8} converges to the subgroup $\Delta$ of $W$ as $k\to \infty$ by Lemma \ref{lem:convergence of subgroups},
and hence the marked Dyer group $(D_k, S)$ converges to $(D, S)$ as $k\to \infty$.

\smallskip
We shall exhibit the idea concretely.
Since any convergent sequence in $\widehat{\N}$ is eventually constant or converges to $\infty$,
by taking a subsequence if necessary,
we assume that any $(i, j)$-th entry $\{m_k(i, j)\}_{k\geq 1}$ of $M_k$ is constant or strictly increasing.
In the latter case,
we assume that $m_k(i, j)\geq 3$.
It follows that the following subsets $N_2, N_p$, and $N_\infty$ of the index set $\{1,\dots,n\}$ do not depend on $k$;
\begin{align*}
N_2&=\Set{i\in{\{1,\dots,n\}}|m_k(i, i)=2},\\
N_p&=\Set{i\in{\{1,\dots,n\}}|3\leq m_k(i, i)<\infty},\\
N_\infty&=\Set{i\in{\{1,\dots,n\}}|m_k(i, i)=\infty}\,.
\end{align*}
Set $n_2=\#N_2, n_p=\#N_p$, and $n_\infty=\#N_\infty$.
By permuting columns and rows of the matrices $M_k$ and $M$,
we may assume that
\[
N_2=\{1,\dots,n_2\},
N_p=\{n_2+1,\dots,n_2+n_p\},
N_\infty=\{n_2+n_p+1,\dots, n_2+n_p+n_\infty=n\}.
\]

Let us denote the ordered vertex set of $\Gamma_k$ and $\Gamma$ by $V=(v_1,\dots,v_n)$.
Since the index subsets $N_2, N_p$, and $N_\infty$ do not depend on $k$,
two vertices $v_i, v_j$ of $\Gamma_k$ and $\Gamma$ with $i\neq j$ are joined by an edge of weight $m_{ij}$ if $i, j\in N_p\sqcup N_\infty$,
and hence the underlying simple graphs of $\Gamma_k$ and $\Gamma$ are the same for any $k$.
We denote the marked Dyer graphs $\Gamma_k$ and $\Gamma$ as $\Gamma_k=(G, f_k, m_k)$ and $\Gamma=(G, f, m)$, respectively.
It follows that $f_k(v_i)\to f(v_i)$ and $m_k(\{v_i, v_j\})\to m(\{v_i, v_j\})$ as $k\to \infty$ for $1\leq i,j\leq n$.
We write the marked Coxeter graphs $\Lambda_k$ (resp. $\Lambda$) for $(L_k, \mu_k)$ (resp. $(L, \mu)$), where $L_k$ (resp. $L$) is a simple graph whose ordered vertex set is $(v_1,\dots,v_n, v'_{n_2+1},\dots,v'_n)$ and two vertices of $L_k$ (resp. L) are joined by an edge of weight $\mu_k$ (resp. $\mu$) in the following manner:
\begin{itemize}
\item[(i)]
The vertices $v_i$ and $v_j$ are joined by an edge of weight $\mu_k(\{v_i ,v_j\})=m_k(\{v_i, v_j\})$ (resp. $\mu(\{v_i, v_j\})=m(\{v_i, v_j\})$) if $\{v_i, v_j\}\in E(G)$.
\item[(ii)]
The vertices $v_i$ and $v'_i$ for $n_2+1\leq i\leq n$ are joined by an edge of weight $\mu_k(\{v_i ,v'_i\})=f_k(v_i)$ (resp. $\mu(\{v_i ,v'_i\})=f(v_i)$).
\end{itemize}
It follows that $L_k=L$ for any $k$ and $\mu_k(\{v, w\})\to \mu(\{v, w\})$ as $k\to \infty$ for any $\{v, w\}\in E(L)$.
By Theorem \ref{theo:Yukita2024 Theorem 3.6},
we obtain that the associated Coxeter system $(W_k, T)$ converges to $(W, T)$ as $k\to \infty$.

\smallskip
We write the ordered free marking $X$ of $\F_{n+n_2+n_\infty}$ as
\[
X=(x_1,\dots,x_n,x'_{n_2+1},\dots,x'_n).
\]
Since the number $n+n_p+n_\infty$ appears many times in the proof,
we write it $m$ for simplicity.
We define an ordered $n$-element subset $Y$ of $\F_m$ by
\[
Y=(y_1,\dots,y_m) \text{ where }y_i=\begin{cases} x_i & (1\leq i\leq n_2) \\ x_ix'_i & (n_2+i\leq i\leq n)\end{cases}.
\]
Set $H=\langle Y\rangle\subset \F_m$.
For the markings $\pi_k:\F_m\twoheadrightarrow W_k$ and $\pi:\F_m\to W$,
consider the marked subgroups $(\Delta_k, \pi(Y))$ and $(\Delta, \pi(Y))$,
where $\Delta_k=\pi_k(H)$ and $\Delta=\pi(H)$.
It follows from Theorem \ref{theo:Soergel2024 Theorem2.8} that $(D_k, S)\cong (\Delta_k, \pi_k(Y))$ and $(D, S)\cong (\Delta, \pi(Y))$ as $n$-marked groups.
Since Lemma \ref{lem:convergence of subgroups} implies that $(\Delta_k, \pi_k(Y))\to (\Delta, \pi(Y))$ as $k\to \infty$,
we see that $(D_k, S)\to (D, S)$ as $k\to \infty$,
and hence the map $\Phi:\mathcal{DM}_n\to \mathcal{D}_n$ is continuous.
This completes the proof, since $\mathcal{DM}_n$ is compact and $\mathcal{D}_n$ is Hausdorff.
\end{proof}

\begin{cor}\label{cor:convergence in Dyer graph}
Convergence in $\mathcal{D}_n$ can be characterized as follows.
Let $\{(D_k, S)\}_{k\ge 1}$ be a sequence of $n$-marked Dyer systems, and let $\Gamma_k=(\Gamma_k, f_k, m_k)$ be the Dyer graph associated with $(D_k, S)$.
Let $(D, S)$ be an $n$-marked Dyer system, and let $\Gamma=(\Gamma, f, m)$ be its associated Dyer graph.
Then $\{(D_k, S)\}_{k\ge 1}$ converges to $(D, S)$ if and only if $\Gamma_k=\Gamma$ as simple graphs for all sufficiently large $k$,
$\displaystyle \lim_{k\to\infty} f_k(v_i) = f(v_i)$ for every vertex $v_i$,
and $\displaystyle \lim_{k\to\infty} m_k(\{v_i,v_j\}) = m(\{v_i,v_j\})$ for every edge $\{v_i,v_j\}$.
\end{cor}

For a Dyer graph $\Gamma=(\Gamma, f, m)$,
set
$V_k(\Gamma)=\Set{v\in V(\Gamma)|f(v)=k}$ for $k\in \widehat{\N}_{\geq 2}$
and
$V_p(\Gamma)=\Set{v\in V(\Gamma)|3\leq f(v)<\infty}$.
Let $(D, S)$ be the Dyer system associated with $\Gamma$.
We write $S_2, S_p$, and $S_\infty$ for the subsets of $S$ corresponding to $V_2, V_p$, and $V_\infty$.
The subgroups of $D$ generated by $S_2, S_p$, and $S_\infty$ are denoted by $D_2, D_p$, and $D_\infty$, respectively.
It follows that the pair $(D_2, S_2)$ is a Coxeter system, and the group $D_p$ (resp. $D_\infty$) is a graph product of finite cyclic groups (resp. a right-angled Artin group).
Let us denote the full subgraph induced by $V_2$ and $V_p\sqcup V_\infty$ by $\Gamma_2$ and $\Gamma_{p, \infty}$, respectively.

\smallskip
A Dyer graph $\Gamma$ is \emph{spherical} (resp. \emph{Euclidean}) if 
(i)
any vertex of $\Gamma_{p, \infty}$ has no edges and
(ii)
the Coxeter system $(D_2, S_2)$ is spherical (resp. Euclidean).
A Dyer system $(D, S)$ is \emph{spherical} (resp. \emph{Euclidean}) if so is the associated Dyer graph.
If $\Gamma$ is spherical or Euclidean,
it is easy to see that $D=D_2\times D_p\times D_\infty$, $D_p$ is finite, and $D_\infty$ is isomorphic to a free abelian group of finite rank.

\begin{lem}\label{lem:subspace of spherical or Euclidean Dyer systems}
The subspace of spherical or Euclidean $n$-marked Dyer systems is closed and open in $\mathcal{D}_n$.
\end{lem}
\begin{proof}
First, we show that such a subspace is closed.
Let $\{(D_k, S)\}_{k\ge1}$ be a convergent sequence of $n$-marked spherical or Euclidean Dyer systems, and let $(D, S)\in \mathcal{D}_n$ be its limit.
It follows from Corollary \ref{cor:convergence in Dyer graph} that the Dyer graph associated with $(D, S)$ satisfies the condition (i) of the definition of spherical and Euclidean Dyer graphs.
By Lemma \ref{lem:convergence of subgroups} and Theorem \ref{theo:clopenness of spherical and Euclidean},
we see that the Coxeter system $(D_2, S_2)$ is spherical or Euclidean,
and hence the assertion follows.

\smallskip
We prove that the subspace is open.
Indeed, for a convergent sequence $\{(D_k, S)\}_{k\ge1}$ of $n$-marked non-spherical and non-Euclidean Dyer systems, by Lemma \ref{lem:convergence of subgroups} and Theorem \ref{theo:clopenness of spherical and Euclidean},
we see that the Coxeter system $(D_2, S_2)$ of the limit $(D, S)$ is non-spherical and non-Euclidean.
\end{proof}

%%%%%%%%%%%%%%%%%%%%%%%%%%%%%%%%%%%%%%%%%%%
%%%%%%%%%%%%%%%%%%%%%%%%%%%%%%%%%%%%%%%%%%%
\section{Monotonicity of the growth rates of Dyer systems}\label{section:5}

In this  section,
we show the monotonicity of growth rates of Dyer systems.
The monotonicity of growth rates of Coxeter systems is verified by Terragni \cite{Terragni2016}.

\smallskip
For two simple graphs $\Gamma$ and $\Lambda$,
a \emph{graph morphism} $\phi:\Gamma\to \Lambda$ consists of maps $\phi_V:V(\Gamma)\to V(\Lambda)$ and $\phi_E:E(\Gamma)\to E(\Lambda)$ such that $\phi_V$ is injective and $\{\phi(v), \phi(w)\}\in E(\Lambda)$ for any edge $\{v, w\}\in E(\Gamma)$.

\begin{defi}
Let $(D, S), (D', S')$ be two Dyer systems, and let $\Gamma=(\Gamma, f, m), \Gamma'=(\Gamma', f', m')$ be the associated Dyer graphs.
Define $(D, S)\preceq (D', S')$ if there exists a graph morphism $\phi:\Gamma\to \Gamma'$ such that $f(v)\leq f'(\phi(v))$ for any $v\in V(\Gamma)$ and $m(\{v, w\})\leq m'(\{\phi(v), \phi(w)\})$ for any $\{v, w\}\in E(\Gamma)$.
\end{defi}
In this section,
we extend the following monotonicity theorem to Dyer systems.
\begin{theo}[{\cite[Theorem A]{Terragni2016}}]
Let $(W, S)$ and $(W', S')$ be Coxeter systems with $(W, S)\preceq (W', S')$.
Then $a_{(W, S)}(m)\leq a_{(W', S')}(m)$ for any $m\in \Z_{\geq 0}$.
In particular,
$\tau(W, S)\leq \tau(W', S')$.
\end{theo}

In order to show the monotonicity of growth rates of Dyer systems,
we recall the solution to the word problem for Dyer groups in \cite{ParisSoergel2023}.

\smallskip
Let $(G, S)$ be a marked group with $S=(s_1,\dots, s_n)$.
We write $\mathcal{S}$ for the subset of $G$ consisting of powers of generators, namely, $\mathcal{S}=\Set{s_i^k\in G|1\leq i\leq n, \ k\in \Z}$,
which is called the \emph{set of syllables of $S$}.
For a generator $s_i$,
a syllable is said to be of \emph{type $s_i$} if it is of the form $s_i^k$.
Since $\mathcal{S}$ also generates $G$,
the word length $\abs{x}_\mathcal{S}$ of $x\in G$ with respect to $\mathcal{S}$ is well-defined,
and is called the \emph{syllabic length of $x$}.
The free monoid on $\mathcal{S}$ is denoted by $\mathcal{S}^\ast$,
in which the product is written as $\boldsymbol{w}\cdot \boldsymbol{w}'$.
We call an element of $\mathcal{S}^\ast$ a \emph{syllabic word on $S$}.
The empty syllabic word is denoted by $\boldsymbol{1}$.
The \emph{degree} $\deg(\boldsymbol{w})$ of a syllabic word $\boldsymbol{w}=(s_1,\dots,s_\ell)\in \mathcal{S}^\ast$ is defined to be $\ell$, $\deg(\boldsymbol{1})=0$ by convention.
Let $\pi_\mathcal{S}:\mathcal{S}^\ast\twoheadrightarrow G$ be the canonical epimorphism.
A syllabic word $\boldsymbol{w}$ is \emph{reduced for $(G, S)$} if $\deg(\boldsymbol{w})=\abs{\pi_\mathcal{S}(\boldsymbol{w})}_\mathcal{S}$.

\smallskip
Let $M$ be an $n\times n$ Dyer matrix,
and let $(D_M, S)$ be the Dyer system associated with $M$,
where $S=(s_1,\dots,s_n)$.
For two syllabic words $\boldsymbol{w}, \boldsymbol{w}'\in \mathcal{S}^\ast$,
\begin{itemize}
\item[(I)]
We say that we go from $\boldsymbol{w}$ to $\boldsymbol{w}'$ through an \emph{elementary $M$-operation of type I},
denoted by $\boldsymbol{w}\xrightarrow{M^{(1)}} \boldsymbol{w}'$,
if $\boldsymbol{w}'$ is obtained from $\boldsymbol{w}$ by concatenating two consecutive syllables of the same type;
$\boldsymbol{w}$ can be written as $\boldsymbol{w}=\boldsymbol{w}_1\cdot (s, t)\cdot \boldsymbol{w}_2$ such that 
$\boldsymbol{w}'=
\begin{cases}
\boldsymbol{w}_1\cdot (st)\cdot \boldsymbol{w}_2 & \text{if }st\neq 1\\
\boldsymbol{w}_1\cdot \boldsymbol{w}_2 & \text{if }st=1
\end{cases}$,
where $\boldsymbol{w}_1, \boldsymbol{w}_2\in S^\ast, s,t\in \mathcal{S}$, and $st\in \mathcal{S}\cup{\{1\}}$.
\item[(II)]
we say that we go from $\boldsymbol{w}$ to $\boldsymbol{w}'$ through an \emph{elementary $M$-operation of type II},
denoted by $\boldsymbol{w}\xrightarrow{M^{(2)}} \boldsymbol{w}'$,
if $\boldsymbol{w}'$ is obtained from $\boldsymbol{w}$ by replacing a syllabic subword $[s, t]_m$ of $\boldsymbol{w}$ into $[t, s]_m$ where $[s, t]_m=(s, t, s, \dots)\in \mathcal{S}^\ast$ with $\deg([s, t]_m)=m$;
$\boldsymbol{w}$ can be written as $\boldsymbol{w}=\boldsymbol{w}_1\cdot [s, t]_m\cdot \boldsymbol{w}_2$ such that 
$\boldsymbol{w}'=\boldsymbol{w}_1\cdot [t, s]_m\cdot \boldsymbol{w}_2$,
where $\boldsymbol{w}_1, \boldsymbol{w}_2\in S^\ast, s,t\in \mathcal{S}, m\geq 2$, and the syllabic word $[s, t]_m$ is reduced for $(D_M, S)$.
\end{itemize}
For two syllabic words $\boldsymbol{w}, \boldsymbol{w}'\in \mathcal{S}^\ast$,
we write $\boldsymbol{w}\stackrel{M}{\leadsto} \boldsymbol{w}'$ if we can go from $\boldsymbol{w}$ to $\boldsymbol{w}'$ through a finite sequence of elementary $M$-operations.
In particular,
if we only use elementary $M$-operations of type I (resp. type II),
then we write $\boldsymbol{w}\stackrel{M^{(1)}}{\leadsto} \boldsymbol{w}'$ (resp. $\boldsymbol{w}\stackrel{M^{(2)}}{\leadsto} \boldsymbol{w}'$).
It follows that $\pi_\mathcal{S}(\boldsymbol{w})=\pi_\mathcal{S}(\boldsymbol{w}')$ for syllabic word $\boldsymbol{w}\leadsto \boldsymbol{w}'$.
We say that a syllabic word $\boldsymbol{w}\in \mathcal{S}^\ast$ is \emph{$M$-reduced} if $\deg(\boldsymbol{w})\leq \deg(\boldsymbol{w}')$ for any syllabic word $\boldsymbol{w}'$ such that $\boldsymbol{w}\stackrel{M}{\leadsto}\boldsymbol{w}'$.

\begin{theo}[{\cite[Theorem 2.2]{ParisSoergel2023}}]\label{theo:solution to the word problem of Dyer system}
Let $M$ be an $n\times n$ Dyer matrix,
and let $(D_M, S)$ be the Dyer system associated with $M$.

(a)
For all $\boldsymbol{w}\in \mathcal{S}^\ast$,
$\boldsymbol{w}$ is reduced for $(D_M, S)$ if and only if $\boldsymbol{w}$ is $M$-reduced.

(b)
For all $\boldsymbol{w}, \boldsymbol{w}'\in \mathcal{S}^\ast$,
if $\boldsymbol{w}$ and $\boldsymbol{w}'$ are both reduced and $\pi_\mathcal{S}(\boldsymbol{w})=\pi_\mathcal{S}(\boldsymbol{w}')$,
then we can go from $\boldsymbol{w}$ to $\boldsymbol{w}'$ through a finite sequence of elementary $M$-operations of type II.
\end{theo}

Given $w\in D_M$,
set $R_M(w)$ to be the set of all the syllabic words reduced for $(D_M, S)$ which are mapped to $w$ by $\pi_\mathcal{S}$,
that is,
\[
R_M(w)=\Set{\boldsymbol{w}\in \mathcal{S}^\ast|\deg(\boldsymbol{w})=\abs{w}_\mathcal{S}, \ \pi_\mathcal{S}(\boldsymbol{w})=w}.
\]
When we give a total order $\leq $ on the set $\mathcal{S}$ of syllables,
the ShortLex order $\leq$ on $\mathcal{S}^\ast$ is defined as follows (see \cite{EpsteinCannonHoltLevyPetersonThurston1992}).
Syllabic words $\boldsymbol{w}<\boldsymbol{w}'$ if $\deg(\boldsymbol{w})<\deg(\boldsymbol{w}')$, or $s_k<s'_k$ in $\mathcal{S}$ where $s_k, s'_k$ are the first syllables of $\boldsymbol{w}, \boldsymbol{w}'$ different from each other.
Since the ShortLex order is a well-ordering,
we choose the minimal syllabic word $\sigma_M(w)\in R_M(w)$ for $w\in D_M$ whenever $\mathcal{S}$ is totally ordered.
It follows that $\sigma_M:D_M\hookrightarrow \mathcal{S}^\ast$ is a section of $\pi_\mathcal{S}:\mathcal{S}^\ast\twoheadrightarrow D_M$.
Note that the choice of $\sigma_M(w)$ depends on the total order on $\mathcal{S}$, but it does not matter in the sequel.
Since both the elementary $M$-operations of types I and II do not increase the degrees of syllabic words,
we see the following corollary from Theorem \ref{theo:solution to the word problem of Dyer system},
which is similar to Corollary 2.4 in \cite{Terragni2016}.

\begin{cor}\label{cor:reduced syllabic word}
Let $M$ be an $n\times n$ Dyer matrix,
and let $(D_M, S)$ be the Dyer system associated with $M$.
Fix a total ordering on $\mathcal{S}$.
Then the following hold for $\boldsymbol{w}, \boldsymbol{w}'\in \mathcal{S}^\ast$ and $w\in D_M$.

(1)
If $\boldsymbol{w}\in R_M(w)$ and $\boldsymbol{w}\stackrel{M}{\leadsto}\boldsymbol{w}'$,
then $\boldsymbol{w}'\in R_M(w)$.

(2)
If $\pi_\mathcal{S}(\boldsymbol{w})=\pi_\mathcal{S}(\boldsymbol{w}')$ and $\boldsymbol{w}'$ is reduced for $(D_M, S)$,
then $\boldsymbol{w}\stackrel{M}{\leadsto} \boldsymbol{w}'$.
\end{cor}

The remainder of this section is devoted to the proof of the following monotonicity of the growth rates of Dyer systems.
\begin{theo}\label{theo:monotonicity of growth}
Let $M$ (resp. $M'$) be $n\times n$ (resp. $n'\times n'$) Dyer matrices with the associated marked Dyer systems $(D_M, S)\preceq (D_{M'}, S')$.
Then $a_{(D, S)}(m)\leq a_{(D', S')}(m)$ for any $m\in \Z_{\geq 0}$.
In particular,
$\tau(D_M, S)\leq \tau(D_{M'}, S')$.
\end{theo}
The proof proceeds as follows.
Fix a graph morphism $\phi:\Gamma\to \Gamma'$ arising from $(D_M, S)\preceq (D_{M'}, S')$,
where $\Gamma$ and $\Gamma'$ are the Dyer graphs associated with $(D_M, S)$ and $(D_{M'}, S')$,
respectively.
We consider the Dyer system $(D_{M''}, S'')$ associated with the image $\phi(\Gamma)$.
Take a sequence $(D_0, S_0),\dots,(D_m, S_m)$ of Dyer systems satisfying the following conditions.
\begin{itemize}
\item[(i)]
$(D_M, S)=(D_0, S_0)$,
$(D_{M''}, S'')=(D_m, S_m)$.
\item[(ii)]
$(D_{k-1}, S_{k-1})\preceq (D_k, S_k)$ for $1\leq k\leq m$.
\item[(iii)]
The Dyer graph $(\Gamma_k, f_k, m_k)$ of $(D_k, S_k)$ is obtained from the Dyer graph $(\Gamma_{k-1}, f_{k-1}, m_{k-1})$ of $(D_{k-1}, S_{k-1})$ by one of the following operations:
increasing the vertex weight $f_{k-1}(v)$ of the unique vertex $v$ to $f_k(v)$,
increasing the edge weight $m_{k-1}(\{v_i, v_j\})$ of the unique edge $\{v_i, v_j\}$ to $m_k(\{v_i, v_j\})$,
or adding to $\Gamma_{k-1}$ the unique edge $\{v_i, v_j\}$ with weight $m_k(\{v_i, v_j\})$.
\end{itemize}
It therefore suffices to verify the assertion for $(D_{k-1}, S_{k-1})\preceq (D_k, S_k)$ and $(D_{M''}, S'')\preceq (D_{M'}, S')$.

\smallskip
We use the following notations:
Let $M$ and $M'$ be $n\times n$ and $n'\times n'$ Dyer matrices.
We denote the associated Dyer systems and graphs by $(D, S), (D', S')$, and $\Gamma=(\Gamma_0, f, m), \Gamma'=(\Gamma'_0, f', m')$, respectively.
Suppose that $(D, S)\preceq (D', S')$.
In view of condition (iii),
we introduce the following elementary operations on Dyer graphs.

\smallskip
\noindent
We say that $\Gamma'$ is obtained from $\Gamma$ by \emph{increasing the vertex weight at $v$} if the underlying simple graphs coincide, the edge weights satisfy $m=m'$,
and the vertex weights agree except at $v$, where $f(v)<f'(v)$.

\smallskip
\noindent
We say that $\Gamma'$ is obtained from $\Gamma$ by \emph{increasing the edge weight at $\{v_i,v_j\}$} if the underlying simple graphs coincide, the vertex weights satisfy $f=f'$,
and the edge weights agree except at $\{v_i,v_j\}$, where $m(\{v_i,v_j\})<m'(\{v_i,v_j\})$.

\smallskip
\noindent
We say that $\Gamma'$ is obtained from $\Gamma$ by \emph{adding the edge $\{v_i,v_j\}$ with weight $m'(\{v_i,v_j\})$} if $V(\Gamma)=V'(\Gamma)$, the vertex weights satisfy $f=f'$,
the edge weights agree on $E(\Gamma)$, 
and $\{v_i,v_j\}\notin E(\Gamma)$ but $\{v_i,v_j\}\in E(\Gamma')$.

\medskip
We write $B_\mathcal{S}(R)$ (resp. $B_{\mathcal{S}'}(R)$) for the set of elements of $D$ (resp. $D'$) whose syllabic lengths are at most $R$,
namely,
\[
B_{\mathcal{S}}(R)=\Set{w\in D|\abs{w}_\mathcal{S}\leq R},\quad
B_{\mathcal{S}'}(R)=\Set{w\in D'|\abs{w}_{\mathcal{S}'}\leq R}.
\]
We fix a graph morphism $\phi:\Gamma_0\to \Gamma'_0$ arising from $(D, S)\preceq (D', S')$.
For each syllable $s_i^k$ of type $s_i$,
we fix a representative so that $\abs{s_i^k}_S=\abs{k}$.
Then, it follows from the inequality $f(v)\leq f'(\phi(v))$ that $\phi$ extends to an injective map $\phi:\mathcal{S}\to \mathcal{S}'$ defined by $\phi(s_i^k)=\phi(s_i)^k$ for $1\leq i\leq n$.
By the universality of the free monoid,
we obtain a monomorphism $\widetilde{\phi}:\mathcal{S}^\ast\to \mathcal{S}'^\ast$ defined as $\widetilde{\phi}(s_{i_1}^{a_1},\dots,s_{i_k}^{a_k})=(\phi(s_{i_1})^{a_1},\dots,\phi(s_{i_k})^{a_k})$.
It follows that $\deg(\widetilde{\phi}(\boldsymbol{w}))=\deg(\boldsymbol{w})$ for any $\boldsymbol{w}\in \mathcal{S}^\ast$, and $\widetilde{\phi}$ sends syllables on $S$ of different types to syllables on $S'$ of different types.
We define a map $\eta:D\to D'$ by $\eta=\pi_{\mathcal{S}'}\circ \widetilde{\phi}\circ \sigma_M$.

\begin{lem}\label{lem:syllabic inclusion}
Suppose that $\Gamma'$ is obtained from $\Gamma$ by increasing the vertex weight at $v$ or the edge weight at $\{v_i, v_j\}$, or adding the edge $\{v_i, v_j\}$ with weight $m'(\{v_i, v_j\})$.
Then the map $\eta:D\to D'$ is injective.
\end{lem}
\begin{proof}
Fix $w_1, w_2\in D$ with $w'=\eta(w_1)=\eta(w_2)\in D'$.
We show that $\sigma_M(w_1)\stackrel{M}{\leadsto}\sigma_M(w_2)$.
Once this is established,
it follows from Corollary \ref{cor:reduced syllabic word} (1) that $\sigma_M(w_2)\in R_M(w_1)$,
and hence $w_2=\pi_{\mathcal{S}}(\sigma_M(w_2))=w_1$.
We therefore proceed to prove that $\sigma_M(w_1)\stackrel{M}{\leadsto}\sigma_M(w_2)$.
Since $(\widetilde{\phi}\circ \sigma_M)(w_i)\in \pi_{\mathcal{S}'}^{-1}(w')$ for $i=1, 2$, and $\sigma_{M'}(w')\in R_{M'}(w')$,
it follows from Corollary \ref{cor:reduced syllabic word} (2) that $(\widetilde{\phi}\circ \sigma_M)(w_i)\stackrel{M'}{\leadsto}\sigma_{M'}(w')$ for $i=1,2$.

\smallskip
First, assume that $\Gamma'$ is obtained from $\Gamma$ by increasing the vertex weight at $v$.
As $\sigma_M(w_i)$ is reduced for $(D, S)$,
Theorem \ref{theo:solution to the word problem of Dyer system} (1) implies that $\sigma_M(w_i)$ is $M$-reduced.
In particular,
any two consecutive syllables of $\sigma_M(w_i)$ are of different types.
Thus,
any two consecutive syllables of $(\widetilde{\phi}\circ \sigma_M)(w_i)$ are of different types since the monomorphism $\widetilde{\phi}$ maps syllables on $S$ of different types to syllables on $S'$ of different types.
It follows that $(\widetilde{\phi}\circ \sigma_M)(w_i)\stackrel{M'^{(2)}}{\leadsto}\sigma_{M'}(w')$ for $i=1,2$.
Any elementary $M'$-operations of type II is invertible,
and hence we obtain that $(\widetilde{\phi}\circ \sigma_M)(w_1)\stackrel{M'^{(2)}}{\leadsto}(\widetilde{\phi}\circ \sigma_M)(w_2)$.
By the assumption that the edge weights $m=m'$,
we see that $\boldsymbol{v}\stackrel{M^{(2)}}{\leadsto}\boldsymbol{v}'$ if and only if $\widetilde{\phi}(\boldsymbol{v})\stackrel{M'^{(2)}}{\leadsto}\widetilde{\phi}(\boldsymbol{v}')$ for any $\boldsymbol{v}, \boldsymbol{v}\in \mathcal{S}^\ast$,
and hence $\sigma_M(w_1)\stackrel{M'^{(2)}}{\leadsto} \sigma_M(w_2)$.

\smallskip
Next, assume that $\Gamma'$ is obtained from $\Gamma$ by increasing the edge weight at $\{v_i, v_j\}$ or adding the edge $\{v_i, v_j\}$.
It follows from the vertex weights $f=f'$ that $\widetilde{\phi}:\mathcal{S}^\ast\to \mathcal{S}'^\ast$ is a bijection.
Suppose that we go from $(\widetilde{\phi}\circ \sigma_M)(w_i)$ to $\sigma_{M'}(w')$ by the following sequence of elementary $M'$-operations.
\[
(\widetilde{\phi}\circ\sigma_M)(w_i)=\boldsymbol{w}_0\xrightarrow{\nu_0}\boldsymbol{w}_1\xrightarrow{\nu_1} \dots \xrightarrow{\nu_{r-2}}\boldsymbol{w}_{r-1}\xrightarrow{\nu_{r-1}}\boldsymbol{w}_r=\sigma_{M'}(w'),
\]
where $\nu_t=M'^{(1)}$ or $M'^{(2)}$ for $0\leq t\leq r-1$.
We shall see that there exists no index $t$ such that $\nu_t=M'^{(2)}$ and
\[
\boldsymbol{w}_{t-1}=\boldsymbol{u}_1\cdot [s'_i, s'_j]_{m'(\{v_i, v_j\})}\cdot \boldsymbol{u}_2,\quad
\boldsymbol{w}_{t}=\boldsymbol{u}_1\cdot [s'_j, s'_i]_{m'(\{v_i, v_j\})}\cdot \boldsymbol{u}_2.
\]
We argue by contradiction.
Set $\boldsymbol{v}_i=\widetilde{\phi}^{-1}(\boldsymbol{w}_i)$ for $0\leq i\leq r$.
Let $t_0$ be the minimum of the indices of such $\nu_i$'s.
Since the elementary $M'$-operations which do not concern $m'(\{v_i, v_j\})$ are the same as the elementary $M$-operations,
we obtain that 
\[
\sigma_M(w_i)=\boldsymbol{v}_0\xrightarrow{M}\boldsymbol{v}_1\xrightarrow{M} \dots \xrightarrow{M}\boldsymbol{v}_{t_0}.
\]
It follows from Corollary \ref{cor:reduced syllabic word} (1) that $\boldsymbol{v}_{t_0}\in R_M(w_i)$.
Since $m(\{v_i, v_j\})<m'(\{v_i, v_j\})$, 
the syllabic word $\boldsymbol{w}_{t_0}=\boldsymbol{u}_1\cdot [s'_i, s'_j]_{m'(\{v_i, v_j\})}\cdot \boldsymbol{u}_2$ can be rewritten as 
\[
\boldsymbol{w}_{t_0}=\boldsymbol{u}_1\cdot [s'_i, s'_j]_{m(\{v_i, v_j\})+1}\cdot \boldsymbol{u}_2',
\]
where $[s'_i, s'_j]_{m'(\{v_i, v_j\})}\cdot \boldsymbol{u}_2=[s'_i, s'_j]_{m(\{v_i, v_j\})+1}\cdot \boldsymbol{u}_2'$.
Therefore,
\begin{align*}
\boldsymbol{v}_{t_0}
&=\widetilde{\phi}^{-1}(\boldsymbol{w}_{t_0})\\
&=\widetilde{\phi}^{-1}(\boldsymbol{u}_1)\cdot [s_i, s_j]_{m(\{v_i, v_j\})+1}\cdot \widetilde{\phi}^{-1}(\boldsymbol{u}_2')\\
&=\widetilde{\phi}^{-1}(\boldsymbol{u}_1)\cdot s_i[s_j, s_i]_{m(\{v_i, v_j\})}\cdot \widetilde{\phi}^{-1}(\boldsymbol{u}_2')\\
&\xrightarrow{M^{(2)}}\widetilde{\phi}^{-1}(\boldsymbol{u}_1)\cdot s_i[s_i, s_j]_{m(\{v_i, v_j\})}\cdot \widetilde{\phi}^{-1}(\boldsymbol{u}_2')\\
&\xrightarrow{M^{(1)}}\widetilde{\phi}^{-1}(\boldsymbol{u}_1)\cdot [s_j, s_i]_{m(\{v_i, v_j\})-1}\cdot \widetilde{\phi}^{-1}(\boldsymbol{u}_2').
\end{align*}
This contradicts the fact that $\boldsymbol{v}_{t_0}$ is a reduced syllabic word.
Hence,
we see that there are no elementary $M'$-operations concerning with $m'(\{v_i, v_j\})$ in the sequence 
\[
(\widetilde{\phi}\circ\sigma_M)(w_i)=\boldsymbol{w}_0\xrightarrow{\nu_0}\boldsymbol{w}_1\xrightarrow{\nu_1} \dots \xrightarrow{\nu_{r-2}}\boldsymbol{w}_{r-1}\xrightarrow{\nu_{r-1}}\boldsymbol{w}_r=\sigma_{M'}(w').
\]
It follows that  $\sigma_M(w_i)\stackrel{M}{\leadsto}\widetilde{\phi}^{-1}(\sigma_{M'}(w'))$ for $i=1,2$.
It follows from Corollary \ref{cor:reduced syllabic word} (1) that $\widetilde{\phi}^{-1}(\sigma_{M'}(w'))\in R_M(w_1)\cap R_M(w_2)$,
and hence $w_1=w_2$.
\end{proof}

Given a word $\boldsymbol{w}$ in the generators $S$,
we write $\chi(\boldsymbol{w})\in \mathcal{S}^\ast$ for the syllabic word obtained from $\boldsymbol{w}$ by compressing consecutive identical generators.
Note that we do not assume that the exponents of each syllable in $\chi(\boldsymbol{w})$ are just the sum of the exponents of corresponding generators in $\boldsymbol{w}$.
For example,
\[
\chi(s_1,s_1,s_1,s_2,s_2^{-1},s_2^{-1},s_3^{-1},s_3^{-1})=(s_1^3, s_2^{-1}, s_3^{-2}).
\]
We define the exponent sum $\exp(\boldsymbol{w})$ of a syllabic word $\boldsymbol{w}=(s_{i_1}^{k_1},\dots,s_{i_m}^{k_m})\in \mathcal{S}^\ast$ by 
\[
\exp(\boldsymbol{w})=\abs{k_1}+\dots+\abs{k_m},
\]
where $k_p$ is chosen so that $\abs{s_{i_p}^{k_p}}_S=\abs{k_p}$.
It follows from the triangle inequality that $\exp(\boldsymbol{w'})\leq \exp(\boldsymbol{w})$ whenever $\boldsymbol{w}\stackrel{M}{\leadsto}\boldsymbol{w}'$.
For a syllabic word $\boldsymbol{w}=(s_{i_1}^{k_1},\dots,s_{i_n}^{k_n})\in \mathcal{S}^\ast$,
by decomposing each syllable $s_{i_p}^{k_p}$ into a product of generator $s_{i_p}$,
we obtain $\abs{\pi_{\mathcal{S}}(\boldsymbol{w})}_S\leq \exp(\boldsymbol{w})$.
For example,
\[
\abs{\pi_\mathcal{S}(s_1^3s_2s_3^{-2})}_S=\abs{s_1s_1s_1s_2s_3^{-1}s_3^{-1}}_S\leq 3+1+2=\exp(s_1^3s_2s_3^{-2}).
\]
\begin{lem}\label{lem:syllabic reduced to word length}
Let $\boldsymbol{w}\in S^\ast$ be a reduced word.
Then $\exp(\chi(\boldsymbol{w}))=\abs{\pi_{(D, S)}(\boldsymbol{w})}$.
In particular,
$\eta(B_{(D, S)}(R))\subset B_{(D', S')}(R)$ for any $R\in \N$.
\end{lem}
\begin{proof}
Write $\boldsymbol{w}=(s_{i_1},\dots,s_{i_m})$ and $\chi(\boldsymbol{w})=(s_{j_1}^{k_1},\dots,s_{j_n}^{k_n})$.
Since $\boldsymbol{w}$ is reduced,
there are no consecutive identical generators in $\boldsymbol{w}$ of the form $ss^{-1}$ and each exponent $k_p$ of $\chi(\boldsymbol{w})$ is chosen so that $\abs{s_{j_p}^{k_p}}_S=\abs{k_p}$.
Therefore,
we obtain 
\[
\abs{\pi_{(D, S)}(\boldsymbol{w})}_S=m=\abs{k_1}+\cdots+\abs{k_n}=\exp(\chi(\boldsymbol{w})).
\]
Fix $w\in B_{(D, S)}(R)$ arbitrarily and let $\boldsymbol{w}$ be a reduced expression of $w$.
It follows from Corollary \ref{cor:reduced syllabic word} (2) that $\chi(\boldsymbol{w})\stackrel{M}{\leadsto}\sigma_M(w)$,
and hence
\[
\abs{\eta(w)}_{S'}
\leq \exp((\widetilde{\phi}\circ \sigma_M)(w))
=\exp(\sigma_M(w))
\leq \exp(\chi(\boldsymbol{w}))=\abs{w}_S=R.
\]
Therefore,
$\eta(B_{(D, S)}(R))\subset B_{(D', S')}(R)$.
\end{proof}

From Lemmas \ref{lem:syllabic inclusion} and \ref{lem:syllabic reduced to word length},
we proceed with the proof of Theorem \ref{theo:monotonicity of growth} as follows.

\begin{proof}[Proof of Theorem \ref{theo:monotonicity of growth}]
We define the Dyer graph $\Gamma''$ as follows.
Its underlying simple graph is the full subgraph of $\Gamma'$ induced by $\phi(V(\Gamma))\subset V(\Gamma')$.
The vertex weight and the edge weight are the restrictions of $f'$ and $m'$ onto the subgraph, respectively.
We write $(D'', S'')$ for the Dyer system associated with $\Gamma''$.

\smallskip
The restriction of the word metric $\abs{\,\cdot\,}_{S'}$ onto $D''$ equals to the word metric $\abs{\,\cdot\,}_{S''}$ on $D''$ by \cite[Lemma 2.5]{ParisSoergel2023}.
It follows that $B_{(D'', S'')}(m)\subset B_{(D', S')}(m)$ for any $m\in \N$.

\smallskip
Since the cardinalities of $V(\Gamma)$ and $V(\Gamma'')$ are the same and $(D, S)\preceq (D'', S'')$,
we can take a sequence $\Gamma_0,\dots,\Gamma_m$ of Dyer graphs such that $\Gamma_0=\Gamma, \Gamma_m=\Gamma''$, and $\Gamma_k$ is obtained from $\Gamma_{k-1}$ by increasing the vertex weight at $v$ or the edge weight at $\{v_i, v_j\}$, or adding the edge $\{v_i, v_j\}$ with weight $m'(\{v_i, v_j\})$.
We denote the Dyer system associated with $\Gamma_k$ by $(D_k, S_k)$.
By Lemma \ref{lem:syllabic inclusion} and \ref{lem:syllabic reduced to word length},
we see that there exists an injective map $\eta_k:B_{(D_{k-1}, S_{k-1})}(m)\to B_{(D_k, S_k)}(m)$ for $m\in \N$,
and hence $a_{(D_{k-1}, S_{k-1})}(m)\leq a_{(D_k, S_k)}(m)$.
This completes the proof.
\end{proof}

%%%%%%%%%%%%%%%%%%%%%%%%%%%%%%%%%%%%%%%%%%%
%%%%%%%%%%%%%%%%%%%%%%%%%%%%%%%%%%%%%%%%%%%
\section{Continuity of the growth rate on $\mathcal{D}_n$}\label{section:6}

\begin{lem}\label{lem:growth of spherical or Euclidean Dyer system}
Let $(D, S)$ be a Dyer system.
Then $\tau(D, S)=1$ if and only if $(D, S)$ is spherical or Euclidean.
\end{lem}
\begin{proof}
Suppose that $(D, S)$ is spherical or Euclidean.
Since $(D, S)=(D_2, S_2)\times (D_p, S_p)\times (D_\infty, S_\infty)$, we obtain that $\tau(D, S)=\max{\{\tau(D_2, S_2), \tau(D_p, S_p), \tau(D_\infty, S_\infty)\}}$.
\begin{itemize}
\item
The growth rate $\tau(D_2, S_2)=1$ because the Coxeter system $(D_2, S_2)$ is spherical or Euclidean.
\item
The growth rate $\tau(D_p, S_p)=1$ because $D_p$ is finite.
\item
The growth rate $\tau(D_\infty, S_\infty)=1$ because $D_\infty$ is isomorphic to a free abelian group of finite rank.
\end{itemize}
Therefore $\tau(D, S)=1$.

\smallskip
Conversely, suppose that $(D, S)$ is non-spherical and non-Euclidean.
\begin{itemize}
\item[(i)]
For the case that there exists a vertex $v$ of $\Gamma_{p, \infty}$ having at least one edge.
Let $\{v, w\}$ be the edge of $\Gamma$.
Since the vertex weight $f(v)\ge 3$,
the edge weight $m(\{v, w\})=\infty$.
Consider the subgroup $H=\langle v, w\rangle$ of $D$ with the marking $(v, w)$.
The subgroup $H$ is isomorphic to the free product $C_{f(v)}\ast C_{f(w)}$,
and hence $(H, (v, w))$ has exponential growth.
\item[(ii)]
For the case that $(D_2, S_2)$ is non-spherical and non-Euclidean Coxeter system.
It follows from \cite{Harpe1987} that $D_2$ contains a free subgroup of rank at least $2$,
and hence $(D_2, S_2)$ has exponential growth.
\end{itemize}
Therefore $\tau(G, S)>1$.
\end{proof}

Now we are in a position to see that the growth rate $\tau:\mathcal{D}_n\to \R_{\geq 1}$ is continuous.
For that,
we use the following formula for the growth series of Dyer systems due to Paris and Varghese \cite{ParisVarghese2024}.

\begin{theo}[{\cite[Theorem 1.1]{ParisVarghese2024}}]\label{theo:Paris Varghese formula}
Let $(D, S)$ be a non-spherical Dyer system.
Then
\[
\dfrac{(-1)^{\#S+1}}{f_{(D, S)}(z)}=\sum_{T\subsetneq S} \dfrac{(-1)^{\#T}}{f_{(D_T, T)}(z)}\,.
\]
\end{theo}
Note that if $(D, S)$ is spherical or Euclidean,
then
\[
f_{(D, S)}(z)=f_{(D_2, S_2)}(z)\cdot f_{(D_p, S_p)}(z)\cdot f_{(D_\infty, S_\infty)}(z).
\]
Set a rational function $F_{(D, S)}(z)=\dfrac{1}{f_{(D, S)}(z)}$.
\begin{lem}\label{lem:convergence of growth series}
Let $\{(D_k, S)\}_{k\ge 1}$ be a convergent sequence of $n$-marked Dyer systems,
and let $(D, S)$ be its limit.

(1)
The rational function $F_{(D_k, S)}(z)$ converges normally to $F_{(D, S)}(z)$ on the unit open disk $\{\abs{z}<1\}$ as $k\to \infty$.

(2)
The set of poles of $F_{(D_k, S)}(z)$ and $F_{(D, S)}(z)$ contained in the open unit disk is locally finite.
\end{lem}
\begin{proof}
(1)
The proof is by induction on $\#S\geq 1$.
Let $\Gamma_k$ and $\Gamma$ be the Dyer graphs associated with $(D_k, S)$ and $(D, S)$, respectively.
By Corollary \ref{cor:convergence in Dyer graph},
we may assume that $\Gamma_k=\Gamma$ as simple graphs,
and write the Dyer graphs $\Gamma_k$ and $\Gamma$ as $\Gamma_k=(G, f_k, m_k)$ and $\Gamma=(G, f, m)$.

\medskip
Suppose that $\#S=1$.
Then $G$ consists of the unique vertex $v$ and Dyer systems are $D_k=\left\langle s \middle| s^{f_k(v)}=1 \right\rangle$ and $D=\left\langle s \middle| s^{f(v)}=1 \right\rangle$.
It follows that $(D_k, S)\cong (C_{f_k(v)}, \{r\})$ and $(D, S)\cong (C_{f(v)}, \{r\})$,
which are the marked cyclic groups as in Example \ref{exa:growth series of cyclic groups}.
If $f(v)<\infty$,
then $(D_k, S)=(D, S)$ for sufficiently large $k$,
and hence the assertion follows.
For the case that $f(v)=\infty$,
we see that $\dfrac{1}{\,f_{(C_p, \{r\})}(z)\,}$ converges normally to $\dfrac{1}{\,f_{(C_\infty, \{r\})}(z)\,}$ on the unit open disk $\{\abs{z}<1\}$ as $p\to \infty$;
indeed the polynomial $f_{(C_{2m+1}, \{r\})}(z)$ equals the first $m+1$ terms of the power series $f_{(C_\infty, \{r\})}(z)$ which has the radius of convergence $1$,
and hence $f_{(C_{2m+1}, \{r\})}$ converges normally to $f_{(C_\infty, \{r\})}(z)$ on the unit open disk $\{\abs{z}<1\}$.
Similarly,
the polynomial $f_{(C_{2m}, \{r\})}(z)+z^m$ converges normally to $f_{(C_\infty, \{r\})}(z)$ on the unit open disk $\{\abs{z}<1\}$,
and hence so does $f_{(C_{2m}, \{r\})}(z)$.

\medskip
Suppose that $\#S=n$ and the assertion holds if $\#S\leq n-1$.

\smallskip
For the case that $(D, S)$ is not spherical and Euclidean,
by Lemma \ref{lem:subspace of spherical or Euclidean Dyer systems},
we assume that $(D_k, S)$ is not spherical and Euclidean for all sufficiently large $k$.
Then the assertion follows from Theorem \ref{theo:Paris Varghese formula} (1) and the inductive hypothesis.

\smallskip
For the case that $(D, S)$ is spherical or Euclidean, 
by Lemma \ref{lem:subspace of spherical or Euclidean Dyer systems},
we assume that $(D_k, S)$ is spherical or Euclidean for all sufficiently large $k$,
and hence we see
\begin{align}
F_{(D_k, S)}(z)&=\dfrac{1}{f_{((D_k)_2, S_2)}(z)}\cdot \dfrac{1}{f_{((D_k)_{p, \infty}, S_p\sqcup S_\infty)}(z)},\label{eq:sequence spherical}\\
F_{(D, S)}(z)&=\dfrac{1}{f_{(D_2, S_2)}(z)}\cdot \dfrac{1}{f_{(D_{p, \infty}, S_p\sqcup S_\infty)}(z)}.\label{eq:limit spherical}
\end{align}
It follows from Corollary \ref{cor:convergence in Dyer graph} that $((D_k)_2, S_2)\to (D_2, S_2)$ and $((D_k)_{p, \infty}, S_{p, \infty})\to (D_{p, \infty}, S_{p, \infty})$ as $k\to \infty$.
The inductive hypothesis implies that each term of the right-hand side of the equation \eqref{eq:sequence spherical} converges normally to that of \eqref{eq:limit spherical}.
This completes the proof of (1).

\medskip
(2)
Let us denote the set of poles of $F_{(D_k, S)}(z)$ and $F_{(D, S)}(z)$ contained in the open unit disk by $R$.
We show that the intersection $R\cap \{\abs{z}<\rho\}$ is finite for any $0<\rho<1$ by induction on $\#S$.

\smallskip
Suppose that $\#S=1$,
that is,
$(D, S)=(C_{f(v)}, \{r\})$ and $(D_k, S)=(C_{f_k(v)}, \{r\})$.
It it enough to see that the absolute value of any zero of the growth polynomial $f_{(C_p, \{r\})}(z)$ is larger than $\rho$ for all sufficiently large $p$.
For $p=2m$,
the growth polynomial satisfies that $(1-z)f_{(C_{2m}, \{r\})}(z)=(1+z)(1-z^m)$,
and hence all zeros of $f_{(C_{2m}, \{r\})}(z)$ are on the unit circle.
Therefore, $f_{(C_{2m}, \{r\})}(z)$ has no zeros on the open disk $\{\abs{z}<\rho\}$.
For $p=2m+1$,
\begin{equation}
(1-z)f_{(C_{2m+1}, \{r\})}(z)=1+z-2z^{m+1}.\label{eq:zero of odd cyclic}
\end{equation}
Let $\zeta\in \{\abs{z}<\rho\}$ be a zero of $f_{(C_{2m+1}, \{r\})}(z)$.
Since $1+\zeta-2\zeta^{m+1}=0$ by equation \eqref{eq:zero of odd cyclic},
together with the inequality $\abs{1+z}>1-\rho$ for $\abs{z}<\rho$,
we have that $1-\rho<2\rho^{m+1}$.
It follows that $f_{(C_{2m+1}, \{r\})}(z)$ has no zeros on the open disk $\{\abs{z}<\rho\}$ for all sufficiently large $m$.

\smallskip
Suppose that $\#S=n$ and the assertion holds if $\#S\leq n-1$.
If $(D, S)$ is not spherical and Eulidean,
then the assertion follows from Theorem \ref{theo:Paris Varghese formula} (1) and the inductive hypothesis.
If $(D, S)$ is spherical or Euclidean,
then the assertion follows from the equalities \eqref{eq:sequence spherical}, \eqref{eq:limit spherical}, and the inductive hypothesis.
\end{proof}

\begin{theo}[{\cite[Hurwitz's Theorem, p.231]{Gamelin2001}}]\label{theo:Hurwitz theorem}
Suppose $\{f_k(z)\}$ is a sequence of analytic functions on a domain $D$ that converges normally to $f(z)$,
and suppose that $f(z)$ has a zero of order $N$ at $z_0$.
Then there exists $\rho>0$ such that for $k$ large,
$f_k(z)$ has exactly $N$ zeros in the disk $\{\abs{z-z_0}<\rho\}$,
counting multiplicity, and these zeros converge to $z_0$ as $k\to \infty$.
\end{theo}

\begin{theo}
The growth rate $\tau:\mathcal{D}_n\to \R_{\geq 1}$ is continuous.
\end{theo}
\begin{proof}
Let $\{(D_k, S)\}_{k\ge 1}$ be a convergent sequence of $n$-marked Dyer systems,
and let $(D, S)$ be its limit.

\medskip
If $(D, S)$ is spherical or Euclidean,
by Lemma \ref{lem:subspace of spherical or Euclidean Dyer systems},
we assume that $(D_k, S)$ is also spherical or Euclidean,
and hence the assertion follows from Lemma \ref{lem:growth of spherical or Euclidean Dyer system}.

\medskip
Suppose that $(D, S)$ is not spherical and Euclidean.
By Corollary \ref{cor:convergence in Dyer graph} and Lemma \ref{lem:growth of spherical or Euclidean Dyer system},
the Dyer system $(D_k, S)$ is not spherical and Euclidean, and the growth rates $\tau(D_k, S), \tau(D, S)>1$.
Set $r=\tau(D, S)^{-1}$ and $r_k=\tau(D_k, S)^{-1}$,
so that $r$ and $r_k$ are zeros of $F_{(D, S)}(z)$ and $F_{(D_k, S)}(z)$, respectively.
In order to obtain a contradiction, suppose that $r_k$ does not converge to $r$. 
Fix $\varepsilon>0$ such that the disk $\{\abs{z-r}<\varepsilon\}$ does not contain $r_k$ for any $k$.
By Lemma \ref{lem:convergence of growth series} (2), 
we assume that any pole of $F_{(D, S)}(z)$ and $F_{(D_k, S)}(z)$ is not contained in the disk $\{\abs{z-r}<\varepsilon\}$.
Then $\{F_{(D_k, S)}(z)\}$ is a sequence of analytic functions on the disk $\{\abs{z-r}<\varepsilon\}$ that converges normally to $F_{(D, S)}(z)$ by Lemma \ref{lem:convergence of growth series} (1).
Theorem \ref{theo:Hurwitz theorem} implies that the disk $\{\abs{z-r}<\varepsilon\}$ contains at least one zero $\zeta_k$ of $F_{(D_k, S)}(z)$ for sufficiently large $k$.
It follows from the triangle inequality that
\begin{equation}
\abs{\zeta_k}\leq \abs{\zeta_k-r}+r<\varepsilon+r. \label{eq:zero bound}
\end{equation}
By Theorem \ref{theo:monotonicity of growth},
we see that $r+\varepsilon\leq r_k$.
Therefore,
together with the inequality \eqref{eq:zero bound},
we obtain that $\abs{\zeta_k}<r_k$.
This contradicts the fact that $r_k$ has the smallest absolute value among the zeros of $F_{(G_k ,S)}(z)$.
\end{proof}

\section*{Acknowledgement}
The author thanks IMUS-María de Maeztu grant CEX2024-001517-M - Apoyo a Unidades de Excelencia María de Maezu for supporting this research, funded by MICIU/AEI/10.13039/501100011033.
This work was supported by JSPS Grant-in-Aid for Early-Career Scientists Grant Number JP25K17258. 

\bibliographystyle{plain}
\bibliography{reference}
\end{document}